\documentclass[12pt,leqno]{amsart}
\usepackage{amsmath}
\usepackage{amsfonts}
\usepackage{amssymb}
\usepackage{graphicx}
\usepackage{hyperref,mathrsfs,nomencl}

\newtheorem{theorem}{Theorem}

\newtheorem{proposition}[theorem]{Proposition}
\theoremstyle{definition}       

\newcommand{\nn}{\mathbb{N}}

\newcommand{\IC}{\mathbb{C}}

\newcommand{\ric}{\operatorname{Ric}}

\newcommand{\loc}{\mathrm{loc}}

\newcommand{\inj}{\mathrm{inj}}

\newcommand{\e}{\epsilon}

\newcommand{\vp}{\varphi}

\begin{document}

\begin{titlepage}\setcounter{page}{1}
\title[$L^{p}$-interpolation inequalities and global Sobolev regularity]{$L^{p}$-interpolation inequalities and  global Sobolev regularity results\\ (with an appendix by Ognjen Milatovic)}

\author[B. G\"uneysu]{Batu G\"uneysu}
\address{Batu G\"uneysu, Institut f\"ur Mathematik, Humboldt-Universit\"at zu Berlin, 12489 Berlin, Germany} \email{gueneysu@math.hu-berlin.de}

\author[S. Pigola]{Stefano Pigola}
\address{Stefano Pigola, Dipartimento di Scienza e Alta Tecnologia - Sezione di Matematica, Università dell'Insubria, 22100 Como, Italy} \email{stefano.pigola@uninsubria.it}
\end{titlepage}

\maketitle 

\begin{abstract} On any complete Riemannian manifold $M$ and for all  $p\in [2,\infty)$, we prove a family of second order $L^{p}$-interpolation inequalities that arise from the following simple $L^{p}$-estimate valid for every $u \in C^{\infty}(M)$:
$$
  \|\nabla u\|_{p}^p  \leq \|u \Delta_{p} u\|_1\in [0,\infty],
$$
where $\Delta_p$ denotes the $p$-Laplace operator. We show that these inequalities, in combination with abstract functional analytic arguments, allow to establish new global Sobolev regularity results for $L^p$-solutions of the Poisson equation for all $p\in (1,\infty)$, and new global Sobolev regularity results for the singular magnetic Schr\"odinger semigroups. 
\end{abstract}

\section{Some definitions from analysis on Riemannian manifolds}

In the sequel, all manifolds are understood to be without boundary and spaces of functions are understood over $\mathbb{R}$. Let $M=(M,g)$ be a smooth connected Riemann $m$-manifold. We denote with $d(x,y)$ the geodesic distance of $x,y\in M$, and for all $r>0$ with $B(x,r)$ the induced open ball with radius $r$ around $x$. We understand all our function spaces like $C^{\infty}(M)$ to be real-valued, while complexifications will be denoted with an index '$\IC$', like $C^{\infty}_{\IC}(M)$ etc.. For $p\in [1,\infty]$ the Banach space $L^p(M)$ is defined with respect to the Riemannian volume measure $\mu$, with $\left\|\cdot\right\|_p$ its norm.\\
Given a smooth $\mathbb{R}$-metric vector bundle $E\to M$, whenever there is no danger of confusion the underlying fiberwise scalar product will be simply denoted with $(\cdot,\cdot)$, with $|\cdot|:=(\cdot,\cdot)^{1/2}$ the induced fiberwise norm. Then one sets
$$
\left\|\Psi\right\|_p:= \left\||\Psi|\right\|_p\quad\text{ for every Borel section $\Psi$ in $E\longrightarrow M $,}
$$
leading to the Banach spaces $\Gamma_{L^p}(M,E)$ and the locally convex spaces $\Gamma_{L^p_{\loc}}(M,E)$ in the usual way. Given another smooth metric $\mathbb{R}$-vector bundle $F\to M$ and a smooth linear partial differential operator $P$ from $E\to M$ to $F\to M$ of order $\leq k$, its adjoint is the uniquely determined smooth linear partial differential operator $P^{\dagger}$ of order $\leq k$ from $F\to M$ to $E\to M$ which satisfies
$$
\int (P^{\dagger}\psi,\phi) d\mu= \int (\psi,P\phi) d\mu 
$$
for all $\psi\in \Gamma_{C^{\infty}}(M,F)$, $\phi\in \Gamma_{C^{\infty}}(M,E)$, with either $\psi$ or $\phi$ compactly supported. Given $f\in \Gamma_{L^1_{\loc}}(M,E)$, this allows to define the validity of $Pf\in  \Gamma_{L^p_{\loc}}(M,E)$ or $Pf\in  \Gamma_{L^p}(M,E)$ in the usual way.\\
As a particular case of the above constructions, we remark that bundles of the form
$$
T^{r,s}M:=(T^*M)^{\otimes^r}\otimes  (T M)^{\otimes^s}\longrightarrow M 
$$
canonically become smooth metric $\mathbb{R}$-vector bundles, in view of the Riemannian structure on $M$. With 
$$
d: C^{\infty}(M)\longrightarrow \Gamma_{C^{\infty}}(M,T^*M)
$$
we denote the total derivative, the gradient can be defined by
$$
\nabla: C^{\infty}(M)\to \Gamma_{C^{\infty}}(M,T^*M),\quad (\nabla u, X):= du (X),
$$
where $X$ is an arbitrary vector field on $M$. The formal adjoint 
$$
\nabla^{\dagger}:\Gamma_{C^{\infty}}(M,T^*M)\longrightarrow C^{\infty}(M)
$$
of $\nabla$ is $(-1)$ times the divergence operator (cf. Theorem 3.14 in \cite{gri}), and with the usual abuse of notation, the Hessian can be defined by
$$
\nabla^2:C^{\infty}(M)\longrightarrow \Gamma_{C^{\infty}}(M,T^{0,2}M),\quad \nabla^2 u(X,Y):= (\nabla^{TM}_X \nabla u, Y),
$$
where $X,Y$ are arbitrary vector fields on $M$, and $\nabla^{TM}$ the Levi-Civita connection on $M$.\\

We further recall that for $p\in [2,\infty)$, the $p$-Laplacian is the nonlinear differential operator defined by
$$
\Delta_p:C^{\infty}(M)\longrightarrow C^{0}(M),\quad \Delta_p u:=\nabla^{\dagger} \ (|\nabla u|^{p-2}\nabla u).
$$
In particular, one finds that $\Delta_2=\Delta:=\nabla^{\dagger}\nabla$ is the usual scalar Laplace-Beltrami operator. \\
Following \cite{babba,GP}, we will call $(\psi_k)_{k\in\mathbb{N}}\subset C^{\infty}_c(M)$ 
\begin{itemize}
\item \emph{a sequence of first order cut-off functions}, if $0\leq \psi_k\leq 1$ pointwise for all $k$, $\psi_k  \nearrow 1$ pointwise, and $\|\nabla \psi_k\|_{\infty}\to 0$ as $k\to\infty$,
\item \emph{a sequence of Hessian cut-off functions}, if it is a sequence of first order cut-off functions such that in addition $\left\|\nabla^2\psi_k\right\|_{\infty}\to 0$ as $k\to\infty$,
\item \emph{a sequence of Laplacian cut-off functions}, if it is a sequence of first order cut-off functions such that in addition $\left\|\Delta\psi_k\right\|_{\infty}\to 0$ as $k\to\infty$.      
\end{itemize}

Note that in view of $|\Delta\psi_k|\leq \sqrt{m} |\nabla^2\psi_k|$, every sequence of Hessian cut-off functions is also a sequence of Laplacian cut-off functions. Moreover, $M$ admits a sequence of first order cut-off functions, if and only if $M$ is geodesically complete, \cite{PS}. The state of the art concerning the existence of Laplacian cut-off functions is contained in \cite{alberto}: there the authors have shown that Laplacian cut-off functions exist on $M$, if $M$ is geodesically complete and there exists a point $o\in M$, and constants $\kappa \in [0,\infty)$, $\widetilde{\kappa}\in [-2,\infty)$, such that
\begin{align}\label{aspopo}
\mathrm{Ric}\geq - \kappa (1+d(\cdot,o)^2)^{-\widetilde{\kappa}/2}.
\end{align}
Furthermore, if $M$ is geodesically complete, then $M$ admits a sequence of Hessian cut-off functions, for example if $M$ has absolutely bounded sectional curvatures \cite{GP}, or if $M$ has a bounded Ricci curvature and a positive injectivity radius \cite{RV}.\vspace{2mm}

Next, we recall that $M$ is said to \emph{satisfy the $L^p$-Calder\'on-Zygmund inequality $CZ(p)$} (where $p\in (1,\infty)$), if there exist constants $C_1\in (0,\infty)$, $C_2\in [0,\infty)$ such that
$$
\left\|\nabla^2 u\right\|_p\leq C_1 \left\|\Delta u\right\|_p+C_2\left\|u\right\|_p\quad\text{ for all $u\in C^{\infty}_c(M)$.}
$$

A simple consequence of Bochner\rq{}s inequality (cf. Appendix \ref{appendix-c}, equation \eqref{bochner}) is that $CZ(2)$ is satisfied if $M$ has Ricci curvature bounded from below by a constant. Moreover, there exist geodesically complete smooth Riemann manifolds which do not satisfy $CZ(2)$ \cite{GP}. The validity of $CZ(p)$ with $p\ne 2$ is a highly delicate business, which has also been addressed in \cite{GP}. For example, $M$ satisfies $CZ(p)$ for every $p\in (1,\infty)$, if $M$ has a positive injectivity radius and a bounded Ricci curvature. For $p\in (1,2]$, using covariant Riesz-transform techniques it is shown in \cite{GP} that $M$ satisfies $CZ(p)$ under geodesic completeness, a $C^1$-boundedness of the curvature, and a rather subtle volume doubling condition (but no assumption on the injectivity radius!).

\section{Main results}

A classical regularity result by Strichartz \cite[Corollary 3.5]{St} states that if $M$ is geodesically complete and if $u,f\in L^2(M)$ and if $u$ is  a solution of the Poisson equation $\Delta u =f$, then one has $\nabla u \in \Gamma_{L^{2}}(M,TM)$. The question we will be concerned in this paper is: \vspace{2mm}

\emph{Are there natural extensions of Strichartz' result at an $L^p$-scale?} \vspace{2mm}

To begin with, we remark that Strichartz' proof for $p=2$ uses Hilbert space arguments, in that it relies on the essential self-adjointness of $\Delta$. In particular, it is clear that the examination of the latter question will require new ideas for $p\ne 2$. In our study of this problem for $p>2$, we found the following very natural result, our first main result:

\begin{theorem}\label{main} Let $M$ be geodesically complete, let $p\in [2,\infty)$ and let $u\in L^{p}(M)\cap C^{\infty}(M)$. Then one has 
\begin{align}\label{E33}
 \|\nabla u\|_{p}^p  \leq \|u \Delta_{p} u\|_1\in [0,\infty],
\end{align}
and, for all 
$$
a_1,a_2,a_3, b_1,b_2,b_3\in [1,\infty]\quad\text{ with }\quad
		1/a_1+1/a_2+1/a_3=1=	1/b_1+1/b_2+1/b_3,
$$
one has
\begin{equation}\label{F31}
 \left\|\nabla u\right\|_p^p  \leq    \left\|u\right\|_{a_1} \left\|\nabla u\right\|_{(p-2)a_3}^{p-2} \left\|\Delta u\right\|_{a_2}+(p-2) \left\|u\right\|_{b_1} \left\|\nabla u\right\|_{(p-2)b_3}^{p-2} \left\|\nabla^2 u\right\|_{b_2}\in [0,\infty].
\end{equation}
\end{theorem} 

The proof of (\ref{E33}) is based on an integration by parts machinery that relies on the existence of a sequence of first order cut-off functions. In particular, our proof is completely different from Strichartz' proof for $p=2$. Then, as we will show, (\ref{F31}) follows straightforwardly from (\ref{E33}) in view of an explicit calculation for the $p$-Laplacian and H\"older's inquality. Inequality (\ref{E33}) itself can be considered as a generalization to $p>2$ of Strichartz' result: indeed, (\ref{E33}) and H\"older's inequality imply that for all smooth $u$ we have $\nabla u \in \Gamma_{L^{p}}(M,TM)$, whenever $u,f \in L^{p}(M)$ and $u$ solves $\Delta_{p} u=f$. Here, $q\in (1,\infty)$ is defined by $1/p+1/q=1$. However a genuine $L^{p}$-extension of Strichartz result is contained in part a) of the following result, which was the main motivation of this paper: 

\begin{theorem}\label{th2} Let $p\in (1,\infty)$, let $f\in L^{p}(M)$, and let $u\in  L^{p}(M)$ be a (distributional) solution of the Poisson equation $\Delta u = f$.\\
\emph{a)} There exists a constant $C \in (0,\infty)$, which only depends on $p$, with the following property: if $M$ is geodesically complete and if 
\begin{equation}\label{HessianAssumption}
\max\{p-2,0\} \nabla^2 u\in \Gamma_{L^p}(M, T^{0,2}M),
\end{equation}
then one has 
 \begin{align}\label{W1p0}
\left\|\nabla u\right\|_p^{2}  \leq   C  \left\|u\right\|_{p}   \left\|f\right\|_{p}+\max\{p-2,0\} \left\|u\right\|_{p}  \left\|\nabla^2 u\right\|_{p}<\infty.
\end{align}
\emph{b)} Assume that $M$ satisfies $CZ(p)$ and admits a sequence of Hessian cut-off functions. Then the following statemens are equivalent:
\begin{itemize}
	\item $\nabla u\in \Gamma_{L^p}(M, TM)$,
	\item $ \nabla^2 u\in \Gamma_{L^p}(M, T^{0,2}M)$.
\end{itemize}
Moreover, there exists a constant $C\in (0,\infty)$, which only depends on $p$ and on the constants from $CZ(p)$, such that if $\nabla u\in \Gamma_{L^p}(M, TM)$ (or equivalently $ \nabla^2 u\in \Gamma_{L^p}(M, T^{0,2}M)$), then one has
\begin{align}\label{insa}
\left\|\nabla u\right\|_p + \left\|\nabla^2 u\right\|_p\leq C \left\|u\right\|_p+ C \left\|f\right\|_p.
\end{align}
\end{theorem}

Concerning part a) of Theorem \ref{th2}: for $p>2$ this result is a simple consequence of (\ref{F31}) and some standard Meyers-Serrin type smoothing argument, while for $p<2$ it relies on an inequality of Coulhon/Duong \cite{CD} for smooth compactly supported functions and a nonstandard smoothing procedure, which is based on a new functional fact proved in Appendix \ref{appendix-a} of this paper: namely, the minimal and maximal $L^p$-realization of $\Delta$ coincide under geodesic completeness (for all $p\in (1,\infty)$), a result that so far was only known under a $C^{\infty}$-boundedness assumption on the geometry of $M$ \cite{Shub, Mi} (which by definition means that the curvature tensor of $M$ and all its derivates are bounded and in addition that $M$ has a positive injectivity radius). Note that, for $1<p\leq 2$, condition \eqref{HessianAssumption} is trivially satisfied hence no $L^{p}$-assumption on the Hessian is required to conclude $\nabla u \in \Gamma_{L^{p}}(M,TM)$. In particular, the case $p=2$ is precisely Strichartz' result.\vspace{2mm}

Concerning part b) of Theorem \ref{th2}: note first that this statement can be considered as partially inverse to part a). In fact, it was proved in \cite{GP}, under the stated assumptions on $M$, that for every $f\in L^{p}(M)$ and every solution $u\in  L^{p}(M)$ of the Poisson equation $\Delta u = f$ with $\nabla u \in \Gamma_{L^{p}}(M,TM)$ one has $\nabla^2 u \in \Gamma_{L^{p}}(M,T^{0,2}M)$, leaving the question open whether the assumption $\nabla u \in \Gamma_{L^{p}}(M,TM)$ was just a technical relict of the proof. Theorem \ref{th2} b) shows that the assumption $\nabla u \in \Gamma_{L^{p}}(M,TM)$ is actually necessary in this context. We   also emphasize that, thanks to the abstract formulation of b), the result is so flexible to provide $L^{p}$ Hessian estimates for the Poisson equation under different geometric conditions on the underlying manifold. We  already  recalled how the validity of $CZ(p)$ and the existence of Hessian cut-off functions can be related to the geometry of the manifold. Concerning the $L^{p}$-integrability of the gradient we mention the interesting paper by E. Amar, \cite{Am}, where  the case of complete manifolds with $\| \ric \|_{\infty}<+\infty$ and $r_{\mathrm{\inj}}(M) >0$ is considered, and the recent preprint by L.-J. Cheng, A. Thalmaier and J. Thompson, \cite{thali}, where the geometric assumptions are strongly relaxed to  $\ric \geq - K^{2}$ for some $K \geq 0$.
Furthermore, we point out that global $W^{2,p}$-estimates of the type \eqref{insa} for solutions of the Poisson equation have been used in \cite{RV} to produce gradient Ricci soliton structures via log-Sobolev inequalities. 
\vspace{2mm}

Finally, we present an application of Theorem \ref{main} concerning the global regularity of the semigroups associated with magnetic Schr\"odinger operators whose potentials are allowed to have local singularities. To this end, we recall that if $M$ is geodesically complete, given an electric potential $0\leq V\in L^2_{\mathrm{loc}}(M)$ and a magnetic potential $A\in \Gamma_{L^4_{\mathrm{loc}}}(M,T M)$ with $\nabla^{\dagger}A\in L^2_{\mathrm{loc}}(M)$, then the magnetic Schr\"odinger operator $\Delta_{A,V}$ in $L^2_{\IC}(M)$, defined initially on $\Psi\in C^{\infty}_{c,\IC}(M)$ by 
\begin{align}\label{mago}
\Delta_{A,V}\Psi:&=(\nabla - \sqrt{-1}A)^{\dagger}(\nabla - \sqrt{-1}A)\Psi\\
&=\Delta \Psi- 2 \sqrt{-1}(A,\nabla \Psi)+\sqrt{-1}(\nabla^{\dagger}A)\Psi+|A|^2\Psi+V\Psi,
\end{align}
is a well-defined nonnegative symmetric operator, which is essentially self-adjoint \cite{GK, LS}. Its self-adjoint closure $H_{A,V}$ is semibounded from below and we can consider its associated magnetic Schr\"odinger semigroup 
$$
 [0,\infty)\ni t\longmapsto \mathrm{e}^{-t H_{A,V}}  \in \mathscr{L}(L^2_{\IC}(M))
$$
defined by the spectral calculus. In fact, a certain self-adjoint extension of $\Delta_{A,V}$ can be defined using quadratic form methods (even without assuming that $M$ is complete), and it is much more convenient to prove \cite{GK,LS} that $C^{\infty}_{c,\IC}(M)$ is an operator core for this extension, rather than proving directly that $\Delta_{A,V}$ is essentially self-adjoint. To do so, the crucial step in the proof is to show the local regularity
\begin{align}\label{addaa2}
\Delta\mathrm{e}^{-t  H_{A,V}}f\in L^2_{\loc,\IC}(M),\quad \nabla \mathrm{e}^{-t H_{A,V}}f\in \Gamma_{L^4_{\loc,\IC}}(M, TM),
\end{align}
for all $f\in L^2_{\IC}(M)$, $t\in(0,\infty)$. This result is needed in the above context to make the machinery of Friedrichs mollifiers work. While the latter local regularity does not need any control on the geometry of $M$, we realized that the inequality (\ref{F31}) from Theorem \ref{main} can be used to answer the following regularity question: \emph{Assume
\begin{align}\label{cv}
0\leq V\in L^2(M),\quad A\in \Gamma_{L^2}(M,TM),\quad  \nabla^{\dagger} A\in L^4(M).
\end{align}
Under which geometric assumptions on $M$ do we have the global regularity}
\begin{align}\label{addaa}
\Delta\mathrm{e}^{-t  H_{A,V}}f\in L^2_{\IC}(M),\quad \nabla \mathrm{e}^{-t H_{A,V}}f\in \Gamma_{L^4_{\IC}}(M, TM)
\end{align}
for all $f\in L^2_{\IC}(M)$, $t\in(0,\infty)$? Towards this aim, we recall that $M$ is called \emph{ultracontractive}\footnote{If $M$ is not geodesically complete, $H_{0,0}$ has to be replaced with the Friedrichs realization of $\Delta$ in the definition of ultracontractivity.}, if the jointly smooth integral of $\mathrm{e}^{-t H_{0,0}}$ satisfies
$$
\sup_{x\in M}\mathrm{e}^{-t H_{0,0}}(x,x)<\infty\quad\text{ for all $t\in (0,\infty)$.}
$$

We are going to use (\ref{F31}) to prove the following result, which seems even new for the Euclidean $\mathbb{R}^m$:

\begin{theorem}\label{uiuiui} Assume $M$ admits a sequence of Laplacian cut-off functions and satisfies $CZ(2)$. Then for all $V$ and $A$ with (\ref{cv}), and all $f\in L^2_{\IC}(M)\cap L^{\infty}(M)$, $t \in(0,\infty)$ one has (\ref{addaa}). If in addition $M$ is ultracontractive, then one has (\ref{addaa}) for all $f\in L^2_{\IC}(M)$, $t\in(0,\infty)$.
\end{theorem}

As we have already observed, $M$ admits a sequence of Laplacian cut-off functions and satisfies $CZ(2)$, if $M$ is geodesically complete with Ricci curvature bounded from below by a constant. If in addition to geodesic completeness and a lower Ricci bound $M$ satisfies the volume non-collapsing condition
$$
\inf_{x\in M}\mu(B(x,r))>0\quad\text{ for all $r\in (0,\infty)$,}
$$
then $M$ is even ultracontractive. This follows from Li-Yau\rq{}s heat kernel estimates, which state that if $M$ is geodesically complete with Ricci curvature bounded from below by a constant, there are constants $C_1,C_2, C_3\in (0,\infty)$ such that
$$
\mathrm{e}^{-t H_{0,0}}(x,y)\leq C_1\mathrm{e}^{tC_2}\mathrm{e}^{-C_3\frac{d(x,y)^2}{t}} \mu(B(x,\sqrt{t})^{-1}\quad\text{ for all $t\in (0,\infty)$}
$$
(with an analogous lower bound).\vspace{2mm}

This paper is organized as follows: in section \ref{haupt} we prove Theorem \ref{main}, section \ref{an1} is devoted to the proof of Theorem \ref{th2}, and section \ref{an2} to the proof of Theorem \ref{uiuiui}. In section \ref{appendix-a} of the appendix the aforementioned result on the equality of the minimal and maximal $L^p$-realization of $\Delta$ under geodesic completeness is proved (cf. Theorem \ref{A}). In section \ref{appendix-b} of the appendix we have recorded a Meyers-Serrin smooting result for Riemannian manifolds, which will be used at several places, and finally section \ref{appendix-c} of the appendix contains a list of standard formulae from calculus on Riemannian manifolds that are used throughout the paper.

\section{Proof of Theorem \ref{main}}\label{haupt}

We first prove the formula
\begin{equation}\label{E5}
\Delta_{p}u =  |\nabla u|^{p-2}\Delta u +(p-2) |\nabla u|^{p-4} \nabla^2u(\nabla u,\nabla u).
\end{equation}
Indeed, one has
\begin{align*}
\Delta_{p}u&= \nabla^{\dagger} \ (|\nabla u|^{p-2}\nabla u)=- (\nabla |\nabla u|^{p-2}, \nabla u )+ |\nabla u|^{p-2}\nabla^{\dagger}\nabla u\\
&=-\big(\nabla (\nabla u,\nabla u)^{p/2-1}, \nabla u \big)+ |\nabla u|^{p-2}\Delta  u\\
&=-(p/2-1)(\nabla u,\nabla u)^{p/2-2}\big(  \nabla (\nabla u,\nabla u), \nabla u \big)+ |\nabla u|^{p-2}\Delta  u\\
&=-(p/2-1)(\nabla u,\nabla u)^{p/2-2}2 (\nabla^{TM}_{\nabla u}\nabla u,\nabla u)+ |\nabla u|^{p-2}\Delta  u\\
&=-(p-2)|\nabla u|^{p-4} \nabla^2u(\nabla u,\nabla u)+ |\nabla u|^{p-2}\Delta  u,
\end{align*}
where we have used (in this order) the product rule, the chain rule, the compatibility of the Levi-Civita connection with the Riemannian metric and, finally, the definition of $\nabla^2$.
Now (\ref{E5}) implies
\begin{equation}\label{E18}
|u\Delta_{p}u| \leq |u||\nabla u|^{p-2}|\Delta u| +(p-2) |u||\nabla u|^{p-2}|\nabla^2 u|,
\end{equation}
so that (\ref{F31}) follows from (\ref{E33}) and H\"older\rq{}s inequality (as $p\geq 2$). \\
It remains to prove (\ref{E33}), fix $0 \leq \vp \in C^{\infty}_{c}(M)$ and define the vector field 
 \begin{equation}\label{E7}
 X= \vp^p u |\nabla u|^{p-2}\nabla u\in \Gamma_{C^1_c}(M,TM).
 \end{equation}
Using again the product rule and the definition of the $p$-Laplacian
we can calculate
\begin{align*}
&\nabla^{\dagger}X= \vp^p u \Delta_p u- (\nabla(\vp^p u),|\nabla u|^{p-2}\nabla u)\\
&= \vp^p u \Delta_p u- u|\nabla u|^{p-2}(\nabla \vp^p ,\nabla u)- \vp^p|\nabla u|^{p}, 
\end{align*}
so that using the divergence theorem we have
 \begin{align}\label{E8}
\int|\nabla u|^{p}\vp^p d\mu= - \int \vp^p u \Delta_{p}u \  d\mu- \int  u ( |\nabla u|^{p-2}\nabla u, \nabla \vp^p  )d\mu=: J_{1} + J_{2}.  
 \end{align}
Clearly, one has
\begin{align}\label{E9}
 |J_{1}| &  \leq \int  \vp^p |u\Delta_{p} u|d\mu.
\end{align}
On the other hand, with $\nabla \vp^p= p \vp^{p-1} \nabla \vp$ and $1/q:=1-1/p$, Young's inequality implies that for all $\e\in(0,\infty)$ we have
\begin{align}\label{E10}
 |J_{2}| &\leq  p \int  \left( |u| |\nabla \vp| \right) \left( \vp |\nabla u|\right)^{p-1} d\mu\\ \nonumber
&\leq \frac{p \e^{q}}{ q}\int  |\nabla u|^{p} \vp^pd\mu + \frac{1}{\e^{p} }\int  |u|^{p} |\nabla \vp|^{p} d\mu\nonumber
\end{align}
Using \eqref{E9}, \eqref{E10} and \eqref{E8} it follows that, for $0<\e <(q/p)^{1/q}$, the term $\int|\nabla u|^{p}\vp^p d\mu$ on the RHS can be absorbed into the LHS and we get:
\begin{align}\label{E11}
 \int |\nabla u|^{p}\vp^p d\mu \leq \frac{1}{1-p\e^q/q}  \int  \vp^p |u\Delta_{p}u| d\mu+\frac{1}{\e^p (1-p\e^q/q)} \int |u|^{p} |\nabla \vp|^{p}d\mu.
\end{align}
As $M$ is geodesically complete, we can pick a sequence $ \vp = \psi_{k} \in C^{\infty}_{c}(M)$ of first order cut-off functions. Taking limits as $k \to  \infty$, using monotone and dominated convergence theorems, and taking $\e\to 0+$ afterwards, we finally obtain the desired estimate (\ref{E33}).

\section{Proof of Theorem \ref{th2}}\label{an1}

a) If $p\geq 2$, and $u \in C^{\infty}(M)$, by applying Theorem \ref{main} b) with $a_{1}=a_{2}= b_{1}=b_{2}=p$ and $a_{3}= b_{3} = p/(p-2)$ we obtain
\begin{equation}\label{p>2}
 \| \nabla u \|_{p}^{2} \leq \| u \|_{p} \| f \|_{p} +(p-2) \| u \|_{p} \|\nabla^{2} u \|_{p}
\end{equation}
which is precisely \eqref{W1p0} with $C =1$. In the general case, by a Meyers-Serrin\rq{}s theorem (cf. Theorem \ref{ms} in Appendix \ref{appendix-b}), we can pick a sequence $(u_k)\subset  C^{\infty}(M)$ with 
$$
\left\|\nabla^2 u_k-\nabla^2 u\right\|_p\to 0,\quad \left\|\Delta u_k- f\right\|_p\to 0, \quad \left\| u_k- u\right\|_p\to 0. 
$$
Then \eqref{p>2} shows that $\nabla u_k$ is a Cauchy sequence in $\Gamma_{L^p}(M,TM)$, which necessarily converges to $\nabla u$. Therefore, evaluating \eqref{p>2} along $u_{k}$ and taking the limit as $k\to +\infty$ completes the proof.\\
If $1<p<2$, and $u \in C^{\infty}_{c}(M)$, by Theorem 4.1 in \cite{CD} we have that
\begin{equation}\label{CD}
 \| \nabla u \|_{p}^{2} \leq C_{p} \| u \|_{p} \|f\|_{p},
\end{equation}
for some absolute constant $C_{p}>0$. This is precisely what is stated in \eqref{W1p0}. In the general case, we appeal to Theorem \ref{A} from Appendix \ref{appendix-a} in order to pick a sequence $(u_k)\subset C^{\infty}_c(M)$ such that 
$$
\quad \left\|\Delta u_k-\Delta u\right\|_p\to 0, \quad \left\| u_k- u\right\|_p\to 0.
$$
By \eqref{CD}, for all $k,h \in \nn$, one has
$$
\left\|\nabla (u_{k}- u_{h}) \right\|_p^{2}  \leq   C_p \left\|u_{k}-u_{h} \right\|_{p}   \left\|\Delta (u_{k} - u_{h})\right\|_{p}.
$$
Whence, we deduce again that $\nabla u_k$ is a Cauchy sequence in $\Gamma_{L^p}(M,TM)$, which necessarily converges to $\nabla u$. To  conclude the validity of \eqref{W1p0} we now evaluate \eqref{CD} along $u_{k}$ and take the limit as $k \to +\infty$.
\\
b) Assume first $\nabla u\in \Gamma_{L^p}(M, TM)$. By Meyers-Serrin we can pick a sequence $(u_k)\subset  C^{\infty}(M)$ with 
$$
\quad \left\|\Delta u_k-  f\right\|_p\to 0, \quad \left\| u_k- u\right\|_p\to 0. 
$$
Then Proposition 3.8 in \cite{GP} implies 
$$
\left\|\nabla^2(u_k-u_{h})\right\|_p\leq C  \left\|\Delta (u_k - u_{h})\right\|_p  + C \| u_{k} - u_{h} \|_{p}
$$
for every $k,h \in \nn$ and for some constant $C>0$ which only depends on the $CZ(p)$ constants. Therefore, with the same Cauchy-sequence argument as above,
$$
\left\|\nabla^2u\right\|_p\leq C \left\| f \right\|_p + C \| u \|_{p} .
$$
If $\max\{p-2,0\} \nabla^2 u\in \Gamma_{L^p}(M, T^{0,2}M)$, then by part a) for every $\epsilon\in (0,\infty)$ we can pick $C_{\epsilon}\in (0,\infty)$ such that
$$
\left\|\nabla u\right\|_p  \leq   C_p \left\|u\right\|_{p}+ C_p   \left\|f\right\|_{p}+\max\{p-2,0\} C_{\epsilon} \left\|u\right\|_{p}+ \epsilon  \left\|\nabla^2 u\right\|_{p}.
$$
Combining these two estimates yields (\ref{insa}).

 \section{Proof of Theorem \ref{uiuiui}}\label{an2}

We start with the following result, which is well-known in the Euclidean case, but has only been recorded so far for smooth magnetic potentials in the case of manifolds:

\begin{proposition}[Kato-Simon inequality] Assume $M$ is geodesically complete and 
\begin{align*} 
0\leq V\in L^2_{\loc}(M),\quad A\in \Gamma_{L^2_{\loc}}(M,TM),\quad  \nabla^{\dagger} A\in L^4_{\loc}(M).
\end{align*}
Then for all $t\in (0,\infty)$, $f\in L^2_{\IC}(M)$, and $\mu$-a.e. $x\in M$ one has
\begin{align*}
\left| \mathrm{e}^{-t H_{A,V}}f(x)\right|\leq \left|\mathrm{e}^{-t H_{0, 0}}f(x)\right|.
\end{align*}
\end{proposition}

\begin{proof} If $A$ is smooth, the asserted inequality follows from Theorem VII.8 in \cite{gubo} (see also \cite{G2}). \\
In the general case, we pick a sequence $(\psi_k)_{k\in\mathbb{N}}\subset C^{\infty}_{c}(M)$ of first order cut-off functions. Then by the Meyers-Serrin theorem, for every $k\in\mathbb{N}$, we can pick a sequence $(A_{k,n})_{n\in\mathbb{N}}\subset \Gamma_{C^{\infty}}(M,TM)$ such that with 
$$
A_k:=\psi_kA
$$
one has
$$
\text{$\lim_{n\to\infty}A_{k,n}= A_k$ in $\Gamma_{L^2}(M,TM)$ and $\lim_{n\to\infty}\nabla^{\dagger}A_{k,n}= \nabla^{\dagger} A_k $ in $L^2(M)$.}
$$
In particular, using (\ref{mago}), for all $\Psi$ in the common operator core $C^{\infty}_{c,\IC}(M)$ of $H_{A_{k,n},V}$ and $H_{A_k,V}$ one has 
$$
\lim_{n\to\infty}\left\|H_{A_{k,n},V} \Psi-H_{A_k,V} \Psi\right\|_2= 0,\quad\text{ so that $\lim_{n\to\infty}\mathrm{e}^{-t H_{A_{k,n},V}}=\mathrm{e}^{-t H_{A_k,V}}$ strongly in $L^{2}_{\IC}(M)$,}
$$
and so
$$
\lim_{n\to\infty}\mathrm{e}^{-t H_{A_{k,n},V}}f(x)=\mathrm{e}^{-t H_{A_k,V}}f(x)\quad\text{ for $\mu$-a.e. $x\in M$,}
$$
possibly by taking a subsequence. \\
Likewise, using the product formula
$$
\nabla^{\dagger} A_k =- (\nabla \psi_k,A )+ \psi_k \nabla^{\dagger}A
$$
one gets
$$
\text{$\lim_{k\to\infty}A_k=  A$ in $\Gamma_{L^2_{\loc}}(M,TM)$ and $\lim_{k\to\infty}\nabla^{\dagger}A_k= \nabla^{\dagger}A$ in $L^2_{\loc}(M)$,}
$$
and so, for all $\Psi$ in the common operator core $C^{\infty}_{c,\IC}(M)$ of $H_{A_k,V}$ and $H_{A,V}$ it holds that
$$
\lim_{k\to\infty}\left\|H_{A_k,V} \Psi-H_{A,V} \Psi\right\|_2= 0,\quad\text{ so that $\lim_{k\to\infty}\mathrm{e}^{-t H_{A_k,V}}=\mathrm{e}^{-t H_{A,V}}$ strongly in $L^{2}_{\IC}(M)$,}
$$
and we arrive at (possibly by taking a subsequence)
$$
\lim_{k\to\infty}\lim_{n\to\infty}\mathrm{e}^{-t H_{A_{k,n},V}}f(x)= \mathrm{e}^{-t H_{A,V}}f(x)\quad\text{ for $\mu$-a.e. $x\in M$.}
$$
This reduces the proof of the Kato-Simon for nonsmooth $A$'s to the aforementioned smooth case.
\end{proof}

\begin{proof}[Proof of Theorem \ref{uiuiui}]

\emph{Step 1: One has 
\begin{align}\label{gn}
 \left\|\nabla u\right\|^2_4  \leq   C_1 \left\| u\right\|_{\infty}  \left\|\Delta  u \right\|_{2}+C_2 \left\| u\right\|_{\infty}  \left\| u \right\|_{2}\quad\text{for all $u\in L^{\infty}(M)\cap L^{2}(M)$ with $\Delta  u\in L^2(M)$,}
\end{align}
for some constants $C_{1},C_{2}>0$ which only depend on the constant from $CZ(2)$.} To see this, we can assume $u$ is real-valued (if not, we decompose $u$ into its real-part and its imaginary-part and use the triangle inequality). We first assume that that $u$ is in addition smooth and pick a sequence $(\psi_k)\subset C^{\infty}_c(M)$ of Laplacian cut-off functions. Then one has (\ref{gn}) with $u$ replaced by $u_k:=\psi_k u$ by Theorem \ref{main} and $CZ(2)$. Using the product rules
$$
\nabla u_k =u\nabla \psi_k+\psi_k\nabla u
$$
and
$$
\Delta u_k = \psi_k \Delta u + u \Delta \psi_k + 2 (\nabla \psi_k, \nabla u)
$$
and that at $u$, $\Delta u$ and $\nabla u$ are $L^2$ (the latter follows, for example, from Theorem \ref{th2} a)), the inequality extends by Fatou and dominated convergence to $u$, taking $k\to \infty$. In the general case, by $u$, $\Delta u\in L^2(M)$ using Meyers-Serrin's theorem we can pick a sequence $(u_k)\subset C^{\infty}(M)$ with $u_k$, $\Delta u_k\in L^2(M)$ with 
$$
\left\|u_k-u\right\|_2\to 0, \quad \left\|\Delta u_k-\Delta u\right\|_2\to 0
$$
 and in addition 
$$
\left\|u_k\right\|_{\infty}\leq \left\|u\right\|_{\infty} \quad\text{ for all $k$.}
$$
 Using (\ref{gn}) with $u_k$ shows that $\nabla u_k$ is Cauchy in $\Gamma_{L^4}(M,TM)$ and then one necessarily has 
$$
\left\|\nabla u_k-\nabla u\right\|_{4}\to 0. 
$$
\emph{Step 2: For all $f\in L^2_{\IC}(M)\cap L^{\infty}(M)$, $t\in(0,\infty)$ one has (\ref{addaa}).} To prove that, we set $f_t:=\mathrm{e}^{-t H_{A,V}}f$ and record that by the Kato-Simon inequality one has the first inequality in 
\begin{align}\label{gn0}
\left\| f_t\right\|_{\infty}\leq \left\|\mathrm{e}^{-t H_{0, 0}}f\right\|_{\infty}\leq \left\|f\right\|_{\infty}<\infty,
\end{align}
where the second inequality follows from noting that
$$
\int \mathrm{e}^{-t H_{0, 0}}(x,y) d\mu (y)\leq 1\quad\text{ for all $x\in M$, $t\in (0,\infty)$,}
$$
as $H_{0, 0}$ stems from a Dirichlet form. Pick now a sequence $(\psi_k)\subset C^{\infty}_c(M)$ of Laplacian cut-off functions. Our aim is to prove 
\begin{align}\label{aspos}
\sup_{k\in \mathbb{N}}\left\|\Delta (\psi_kf_t)\right\|_2<\infty.
\end{align}
Indeed, then $(\Delta (\psi_kf_t))_k$ has a subsequence which converges weakly to some $h\in L^2_{\IC}(M)$, but as we have $\left\|\psi_kf_t- f_t\right\|_2\to 0$, we have $\Delta f_t=h\in L^2_{\IC}(M)$. Then, applying (\ref{gn}) with $u=f_t$ using (\ref{gn0}) also shows $\nabla f_t\in \Gamma_{L^4_{\IC}}(M,TM)$.\\
Thus it remains to prove (\ref{aspos}): To this end, by the spectral calculus we have 
$$
\mathrm{Dom}(H_{A,V})\subset \mathrm{Dom} (\sqrt{H_{A,V}})
$$
and $f_t\in \mathrm{Dom}(H_{A,V})$, and from essential self-adjointness
$$
\mathrm{Dom}(H_{A,V})=\{u\in L^2_{\IC}(M): \Delta_{A,V}u\in  L^2_{\IC}(M)\},\>\>H_{A,V}u=\Delta_{A,V}
$$
and
$$
\mathrm{Dom}(\sqrt{H_{A,V}})=\{u\in L^2_{\IC}(M): (\nabla-\sqrt{-1}A)f\in \Gamma_{L^2_{\IC}}(M, TM),\>\sqrt{V}f\in L^2_{\IC}(M) \}.
$$
It follows from a simple calculation that $\psi_kf_t\in \mathrm{Dom}(H_{A,V})$ with
\begin{align}\label{pqa}
H_{A,V}(\psi_k f_t)=\psi_k \Delta_{A,V}f+ 2\big(\nabla \psi_k ,(\nabla-\sqrt{-1}A) f_t\big) -(\Delta\psi_k) f_t.
\end{align}
On the other hand, from $(\psi_k f_t)\in \mathrm{Dom}(\sqrt{H_{A,V}})$ we have 
$$
(\nabla-\sqrt{-1}A)(\psi_k f_t)\in \Gamma_{L^2_{\IC} }(M,TM)
$$
which from the assumption on $A$ easily implies
\begin{align}\label{aaspopo}
\nabla(\psi_k f_t)=(\nabla-\sqrt{-1}A)(\psi_k f_t)-\sqrt{-1}A(\psi_k f_t)\in \Gamma_{L^2_{\IC}}(M,TM)
\end{align}
as $\psi_k f_t$ is bounded with a compact support. Likewise, it follows from (\ref{aaspopo}) and the assumptions on $A$ and $V$ that 
\begin{align*}
\Delta (\psi_k f_t)=&\Delta_{A,V}(\psi_k f_t)+ 2 (A,\nabla (\psi_k f_t))\\
&-\sqrt{-1}(\nabla^{\dagger}A) \psi_k f_t-|A|^2 \psi_k f_t-V \psi_k f_t \in L^2_{\IC}(M),
\end{align*}
so that 
\begin{align*}
\left\|\Delta (\psi_k f_t)\right\|_2\leq&\left\|\Delta_{A,V}(\psi_k f_t)\right\|_2+ 2\left\| (A,\nabla (\psi_k f_t))\right\|_2+
\left\|((\nabla^{\dagger}A)  -|A|^2 -V) \psi_k f_t\right\|_2 \\
\leq & \left\|\psi_k \Delta_{A,V}f\right\|_2+2\left\|\big(\nabla \psi_k ,(\nabla-\sqrt{-1}A) f_t\big)\right\|_2+\left\|(\Delta\psi_k) f_t\right\|_2\\
&+ 2\left\| (A,\nabla (\psi_k f_t))\right\|_2+\left\|((\nabla^{\dagger}A)  -|A|^2 -V) \right\|_2\left\|f\right\|_{\infty} \\
\leq & \left\| \Delta_{A,V}f\right\|_2+2\sup_k\left\| \nabla \psi_k\right\|_{\infty}\left\| (\nabla-\sqrt{-1}A) f_t \right\|_2+\sup_k\left\|(\Delta\psi_k)\right\|_{\infty}\left\| f_t\right\|_2\\
&+ 2\left\| (A,\nabla (\psi_k f_t))\right\|_2+\left\|((\nabla^{\dagger}A)  -|A|^2 -V) \right\|_2\left\|f\right\|_{\infty}.  
\end{align*}
Finally, using (\ref{gn}), for every $\epsilon \in (0,\infty)$ we have
$$
\left\| (A,\nabla (\psi_k f_t))\right\|_2\leq \left\| A\right\|_4\left\|\nabla (\psi_k f_t)\right\|_4\leq \left\| A\right\|_4
C(1/\epsilon)\left\|f_t\right\|_{\infty}+ \left\| A\right\|_4 C\epsilon  \left\|\Delta (\psi_k f_t)\right\|_2 ,
$$
completing the proof of (\ref{aspos}).\vspace{2mm}

\emph{Step 3: Removal of the assumption $f\in L^{\infty}(M)$ in the ultracontractive case.} If in addition $M$ is ultracontractive, then for all $s\in(0,\infty)$ one has that $\mathrm{e}^{-s H_{0,0}}$ maps $L^2_{\IC}(M)\to L^{\infty}_{\IC}(M)$, so by (\ref{gn0}) the same is true for $\mathrm{e}^{-s H_{A,V}}$. Thus, for all $f\in L^{2}_{\IC}(M)$, $t\in(0,\infty)$, one has
$$
\mathrm{e}^{-t H_{A,V}}f=  \mathrm{e}^{-(t/2) H_{A,V}}\tilde{f},
$$
where $\tilde{f}:=\mathrm{e}^{-(t/2) H_{A,V}}f\in L^{2}_{\IC}(M)\cap L^{\infty}(M)$, $t\in(0,\infty)$, so the claim follows from Step 2. This completes the proof.

\end{proof}

\appendix

\section{ $\Delta_{\min, p}= \Delta_{\max, p}$ under geodesic completeness\\ (by Ognjen Milatovic)}\label{appendix-a}

Let $M$ be a smooth Riemannian manifold. Given a linear partial differential operator
$$
\mathscr{T}:C^{\infty}(M)\longrightarrow C^{\infty}(M)
$$
with smooth coefficients and $p\in (1,\infty)$, we define a closable operator $\mathscr{T}_p$ in $L^p(M)$ as follows:
$$
\mathrm{Dom}(\mathscr{T}_p)=C^{\infty}_c(M),\quad  \mathscr{T}_pf:=\mathscr{T}f.
$$
Then one can further define two closed extensions in $L^p(M)$ of $\mathscr{T}_p$ as follows: $\mathscr{T}_{\min, p}$ is defined as the closure of $\mathscr{T}_p$ in $L^p(M)$, and $\mathrm{Dom}(\mathscr{T}_{\max, p})$ is defined to be the space of all $f\in L^p(M)$ such that $\mathscr{T} f\in L^p(M)$ (distributionally), with $\mathscr{T}_{\max, p}f:= \mathscr{T} f$ for such $f$\rq{}s. Assuming $M$ has a $C^{\infty}$-bounded geometry it has been shown in \cite{Shub,Mi} that $\Delta_{\min, p}= \Delta_{\max, p}$. The main result of this section shows that in fact one can completely remove any curvature and injectivity radius for the equality $\Delta_{\min, p}= \Delta_{\max, p}$:

\begin{theorem}\label{A} Let $M$ be geodesically complete and let $p\in (1,\infty)$. Then one has $\Delta_{\min, p}= \Delta_{\max, p}$, in other words, $C^{\infty}_c(M)$ is an operator core for $\Delta_{\max, p}$. Moreover, $\Delta_{\max, p}$ generates a strongly continuous contraction semigroup in $L^p(M)$.
\end{theorem}

\begin{proof} It follows from distribution theory (cf. Lemma I.25 in \cite{gubo}) that under the isometric identification $L^p(M)=L^{q}(M)^*$ (where $q\in (1,\infty)$ is defined by $1/p+1/q=1$), the adjoint $(\mathscr{T}_{\min,q})^*$ for every $\mathscr{T}$ as above is given by $(\mathscr{T}_{\min,q})^*=(\mathscr{T}^{\dagger})_{\max,p}$, where 
$$
\mathscr{T}^{\dagger}:C^{\infty}(M)\longrightarrow C^{\infty}(M)
$$
denotes the formal adjoint of $\mathscr{T}^{\dagger}$. In particular, $(\Delta_{\min,q})^*= \Delta_{\max,p}$. It has been shown in \cite{St} that under geodesic completeness $\Delta_{\min,r}$ is the generator of a  strongly continuous contraction semigroup in $L^r(M)$ for all $r\in (1,\infty)$. As adjoints of generators of strongly continuous contraction semigroups in reflexive Banach spaces again generate such semigroups (\cite{ABHN}, p. 138), this property remains true for $\Delta_{\max,p}$. As generators of  strongly continuous contraction semigroups are maximally accretive (\cite{RS}, p. 241), it follows that $\Delta_{\max,p}$ is an accretive extension of the maximally accretive operator $\Delta_{\min,p}$ and so $\Delta_{\min,p}=\Delta_{\max,p}$.  
\end{proof}

Note that $\Delta_{\min, p}= \Delta_{\max, p}$ is equivalent to the following density result: For every $f\in L^p(M)$ (w.l.o.g smooth by Meyers-Serrin; cf. Theorem \ref{ms} below) with $\Delta f\in L^p(M)$ there exists a sequence $(f_k)\subset C^{\infty}_c(M)$ such that as $k\to\infty$,
$$
\left\|f_k-f\right\|_p\to 0,\quad \left\|\Delta f_k-\Delta f\right\|_p\to 0.
$$
It is remarkable that even assuming the existence of Laplacian cut-off functions, there seems to be no way to prove this density by hand, that is, without using some functional analytic machinery. In fact, this \lq\lq{}phenomenon\rq\rq{} already occurs for $p=2$.\vspace{2mm}

Let $p\in (1,\infty)$. The paper \cite{Mi} deals with operator core problems as in Theorem \ref{A} in the situation where $\Delta$ is replaced with the Schr\"odinger operator $\Delta+V$ with $0\leq V\in L^p_{\mathrm{loc}}(M)$. In fact, the main result therein shows that $C^{\infty}_c(M)$ is an operator core for $\Delta_{V,\max, p}$, if $M$ has a $C^{\infty}$-bounded geometry, and if $\Delta_{V,\max, p}$ is the closed operator in $L^p(M)$ defined by
$$
\mathrm{Dom}(\Delta_{V,\max, p}):= \{f\in L^p(M): Vf\in L^1_{\mathrm{loc}}(M) ,(\Delta+V)f\in L^p(M)\}.
$$
The proof given there uses the $C^{\infty}$-boundedness assumption on $M$ only to prove that $M$ is $L^p$-positivity preserving in the language of \cite{babba} and that $\Delta_{\max, p}=\Delta_{\min, p}$, together with some perturbation theory. As by recent results it is known that geodesically complete Riemannian manifolds with a Ricci curvature bounded from below by a constant are $L^q$-positivity preserving (in fact also for $p=1$ and $p=\infty$) \cite{babba,alberto}, showing the following result which should be of an independent interest: 

\begin{theorem} Let $M$ be geodesically complete with a Ricci curvature bounded from below by a constant, and let $p\in (1,\infty)$, $0\leq V\in L^p_{\mathrm{loc}}(M)$. Then $C^{\infty}_c(M)$ is an operator core for $\Delta_{V,\max, p}$.
\end{theorem}

It is also reasonable to expect that using the techniques fro \cite{Mi2}, these results can be extended to covariant Schr\"odinger operators.

\section{A geometric Meyers-Serrin Theorem}\label{appendix-b}

The following result follows from the main result in \cite{GGP} and its proof:

\begin{theorem}\label{ms} Let $M$ be a smooth Riemannian manifold, let $E\to M$ be a smooth metric $\mathbb{K}$-vector bundle (where $\mathbb{K}\in \{\mathbb{R},\mathbb{C}\}$), and let $p\in (1,\infty)$. Then for every $f\in \Gamma_{L^p}(M,E)$ there exists a sequence $(f_k)\subset \Gamma_{C^{\infty}}(M,E)$, whose elements can be chosen compactly supported if $f$ has a compact ($\mu$-essential) support, such that 
\begin{itemize}
\item $\left\|f_k-f\right\|_p\to 0$ as $k\to \infty$,
\item $\left\|f_k\right\|_{\infty}\leq \left\|f \right\|_{\infty} \in [0,\infty]$ for all $k$,
\item for every smooth metric vector bundle $F \to M$ over $\mathbb{K}$, every $l\in \mathbb{N}_{\geq 1}$, and every smooth $\mathbb{K}$-linear partial differential operator $P$ from $M\to E$ to $M\to F$ of order $\leq l$ with $Pf \in \Gamma_{L^p}(M,F)$, one has $\left\|Pf_k-Pf\right\|_p\to 0$ as $k\to \infty$, if in case $l\geq 2$ one has $f \in \Gamma_{W^{l-1,p}_{\mathrm{loc}}}(M,E)$ (with no further assumption for $l=1$).
\end{itemize}
\end{theorem}

\section{Some useful formulae}\label{appendix-c}

Let us first record that for all vector fields $X$, $Y$, $Z$ on $M$ one has
\[
X(Y,Z)= (\nabla^{TM}_XY,Z)+(Y,\nabla^{TM}_XZ),
\]
where in the LHS $X$ acts as a derivation on the smooth function $x\mapsto (X(x),Y(x))$ on $M$. This equation just means that the Levi-Civita connection is compactible with the Riemannian metric. Assume $\phi_1$ is a function on $M$. Recalling that $\nabla^{\dagger}$ is $(-1)$ times the divergence operator, one finds the product rule

\[
\nabla^{\dagger} (\phi_1 Y)= \phi_1 \nabla^{\dagger} Y- (\nabla \phi_1, Y).
\]
If $\phi_2$ is another function on $M$, then one has the product rule
\[
\nabla (\phi_1\phi_2)= \phi_1\nabla\phi_2+\phi_2\nabla\phi_1,
\]
and
\[
\Delta (\phi_1\phi_2) = \phi_1 \Delta \phi_2 + \phi_2 \Delta \phi_1 + 2 (\nabla \phi_1, \nabla \phi_2).
\]

For every function $f$ on $\mathbb{R}$ one has the chain rule
\[
\nabla f(\phi_1)=  f\rq{}( \phi_1 )\nabla \phi_1.
\]

If $X$ is compactly supported, then the divergence theorem holds
\[
\int \nabla^{\dagger}X d\mu=0,
\]
which holds by the definition of $\nabla^{\dagger}$:
$$
 \int \nabla^{\dagger}X d\mu = \int (\nabla^{\dagger}X ) \cdot 1 d\mu = \int (\nabla^{\dagger}X ) \cdot 1 d\mu= \int X \cdot (\nabla 1) d\mu =0.
$$

Finally, we record Bochner's equality: 
\begin{align*}
\left|\nabla^2 \phi_1\right|^2=-\frac{1}{2}\Delta |\nabla \phi_1|^2+(\nabla \phi_1,\nabla \Delta \phi_1) - \mathrm{Ric}(\nabla \phi_1,\nabla \phi_1).
\end{align*}
In particular, it follows that if $\mathrm{Ric}\geq -C$ for some constant $C\geq 0$ and $\phi_1$ is compactly supported, then in view of $\Delta=\nabla^{\dagger}\nabla$ one has
\begin{align}\label{bochner}
\int \left|\nabla^2 \phi_1\right|^2 d\mu &\leq -\int\frac{1}{2}(\nabla^{\dagger}\nabla |\nabla \phi_1|^2)\cdot 1 d\mu+ \int(\nabla \phi_1,\nabla \Delta \phi_1)d\mu+C \int (\nabla \phi_1,\nabla \phi_1) d\mu\\
& =   \int | \Delta \phi_1|^2d\mu+C \int (\nabla \phi_1,\nabla \phi_1) d\mu,
\end{align}
which is nothing but the Calder\'on-Zygmund inequality $CZ(2)$.

\subsection*{Acknowledgments} The authors are grateful to the anonymous referee for a careful reading of the manuscript and for valuable remarks. The second named author is partially supported by the Italian group INdAM-GNAMPA.


\begin{thebibliography}{9999} 

\bibitem [Am]{Am} E. Amar, \textit{On the $L^{r}$ Hodge theory in complete non compact Riemannian manifolds}. Math. Z., \textbf{287} (2017), 751--795.

\bibitem[BS] {alberto} D. Bianchi, A. Setti, \emph{Laplacian cut-offs, fast diffusions on manifolds and other applications.} Calc. Var. (to appear). \url{https://doi.org/10.1007/s00526-017-1267-9}

\bibitem[ABHN]{ABHN}  W. Arendt, C.J.K Batty, Charles J. K., M. Hieber, F. Neubrander, \emph{Vector-valued Laplace transforms and Cauchy problems.} Monographs in Mathematics, \textbf{96}. Birkh\"auser Verlag, Basel, 2001.

\bibitem[CTT]{thali} Chen, L.J. \& Thalmaier, A. \& Thompson, A.: \emph{Quantitative $C^1$ estimates by Bismut formula.} Preprint (2017).  \url{https://arxiv.org/pdf/1707.07121.pdf}

\bibitem[CD]{CD} T. Coulhon, X.T. Duong: \emph{ Riesz transform and related inequalities on noncompact Riemannian manifolds.} Comm. Pure Appl. Math. \textbf{56} (2003), no. 12, 1728--1751.

\bibitem[GGP] {GGP} D. Guidetti, B. G\"uneysu, D. Pallara: {\it $L^1$-elliptic regularity and $H=W$ on the whole $L^p$-scale on arbitrary manifolds,}  Annales Academiae Scientiarum Fennicae, Mathematica (2017) Volumen 42, 497--521.  

\bibitem [GP] {GP} B. G\"uneysu, S. Pigola, \textit{The Calder\'on-Zygmund inequality and Sobolev spaces on noncompact Riemannian manifolds}. Adv. Math. \textbf{281} (2015), 353--393.

\bibitem [Gue] {babba} B. G\"uneysu, \emph{ Sequences of Laplacian cut-off functions.} J. Geom. Anal. \textbf{26} (2016), 171--184. 

\bibitem [Gue2] {G2} B.  G\"uneysu, \emph{On generalized Schr\"odinger semigroups.} J. Funct. Anal. \textbf{262} (2012), 4639--4674.

\bibitem[Gue3] {gubo} B. G\"uneysu, \emph{Covariant Schr\"odinger semigroups on noncompact Riemannian manifolds}. \emph{Operator Theory: Advances and Applications}, 264, Birkh\"auser, 2017.

\bibitem [Gri] {gri} Grigor'yan, A.: \emph{Heat kernel and analysis on manifolds.}
AMS/IP Studies in Advanced Mathematics, 47. American Mathematical Society, Providence, RI; International Press, Boston, MA, 2009. 

\bibitem [GK]{GK} R. Grummt, M. Kolb, \emph{ Essential selfadjointness of singular magnetic Schr\"odinger operators on Riemannian manifolds.} J. Math. Anal. Appl. \textbf{388} (2012), 480--489.

\bibitem[LS]{LS} H. Leinfelder, C.G. Simader, \emph{Schr\"odinger Operators with Singular Magnetic Vector Potentials.} Math. Z. (1981)  Volume: 176, page 1--19.


\bibitem [Mi]{Mi} O. Milatovic, \emph{On $m$-accretive Schr\"odinger operators in $L^p$-spaces on manifolds of bounded geometry.} J. Math. Anal. Appl. \textbf{324} (2006),  762--772.

\bibitem [Mi2]{Mi2} O. Milatovic, \emph{On $m$-accretivity of perturbed Bochner Laplacian in $L^p$ spaces on Riemannian manifolds.} Integral Equations Operator Theory \textbf{68} (2010), 243--254.

\bibitem[PS]{PS} S. Pigola, A.G. Setti, \textit{Global divergence theorems in nonlinear PDEs and geometry}. Ensaios Matem\'aticos, \textbf{26}. Sociedade Brasileira de Matem\'atica, Rio de Janeiro, 2014.

\bibitem[RS] {RS} M. Reed, B. Simon,  \emph{ Methods of modern mathematical physics. II. Fourier analysis, self-adjointness.} Academic Press [Harcourt Brace Jovanovich, Publishers], New York-London, 1975. 

\bibitem [RV] {RV} M. Rimoldi, G. Veronelli, \textit{Extremals of Log Sobolev inequality on non-compact manifolds and Ricci soliton structures}. Preprint (2016). Available at \url{https://arxiv.org/pdf/1605.09240.pdf}.

\bibitem[Sh]{Shub} M.A. Shubin: \emph{Spectral Theory Of Elliptic Operators On Non-Compact Manifolds}.Asterisque (207), 1992, 35--108.

\bibitem [St]{St} R. Strichartz, \textit{Analysis of the Laplacian on the complete Riemannian manifold}. J. Funct. Anal. \textbf{52} (1983), 48--79.



\end{thebibliography}
\end{document}